\documentclass[12pt]{article}
\usepackage[margin=1in]{geometry}
\usepackage{amsthm, amsmath,amsfonts,amssymb,euscript,hyperref,graphics,color,slashed,mathrsfs}
\usepackage{graphicx}
\usepackage{comment}
\usepackage{import}
\usepackage{tikz}
\usepackage{latexsym}
\usepackage{mathtools}
\usepackage{appendix}
\usepackage{stmaryrd}
\usetikzlibrary{arrows}
\usepackage{pgfplots}
\usepackage{extarrows}
\usepackage{url}
\usepackage[utf8]{inputenc}
\usepackage[vietnamese,main=english]{babel}

\makeatletter
\newcommand{\subjclass}[2][2020]{%
  \let\@oldtitle\@title%
  \gdef\@title{\@oldtitle\footnotetext{#1 \emph{Mathematics subject classification.} #2}}%
}
\newcommand{\keywords}[1]{%
  \let\@@oldtitle\@title%
  \gdef\@title{\@@oldtitle\footnotetext{\emph{Key words and phrases.} #1}}%
}
\makeatother
\let\le\leqslant
\let\ge\geqslant
\let\eps\varepsilon
\let\p\partial
\renewcommand{\d}{\mathrm{d}}
\renewcommand\pmod[1]{\;(\operatorname{mod}#1)}
\newtheorem{theorem}{Theorem}[section]
\newtheorem*{theorem*}{Theorem}
\newtheorem{lemma}[theorem]{Lemma}

\theoremstyle{remark}
\newtheorem{definition}[theorem]{\bf Definition}
\newtheorem{remark}[theorem]{\bf Remark}

\newtheorem*{acknowledgements}{\bf Acknowledgements}
\numberwithin{equation}{section}

\begin{document}
\title{A note on the irrationality of $\zeta_2(5)$}

\author{Li Lai, Johannes Sprang, Wadim Zudilin}
\date{}

\maketitle

\begin{abstract}
In a spirit of Ap\'ery's proof of the irrationality of $\zeta(3)$,
we construct a sequence $p_n/q_n$ of rational approximations to the $2$-adic zeta value $\zeta_2(5)$ which satisfy $0 <|q_n\zeta_2(5)-p_n|_2 < \max\{ |p_n|,|q_n| \}^{-1-\delta}$ for an explicit constant $\delta>0$.
This leads to a new proof of the irrationality of $\zeta_2(5)$, the result established recently by Calegari, Dimitrov and Tang using a different method.
Furthermore, our approximations allow us to obtain an upper bound for the irrationality measure of this $2$-adic quantity; namely, we show that $\mu(\zeta_2(5)) \le (16\log 2)/(8 \log 2 -5) = 20.342\dots$\,.
\end{abstract}

\section{Introduction}

Ap\'ery's proof \cite{Ape1979} of the irrationality of
\[
\zeta(3)=\sum_{k=1}^\infty\frac1{k^3}
\]
remains a significant attraction in number theory.
One reason for this is that the rational approximations to $\zeta(3)$ used in the proof link with many other mathematics areas\,---\,combinatorics, analysis, algebraic geometry, differential equations, integrable models, mathematical physics\,---\,to name a few.
Another reason behind this attractiveness is a recognised difficulty of proving anything beyond $\zeta(3)\notin\mathbb Q$ for the values of Riemann's zeta function and its numerous generalisations that does not follow from known transcendence results for $\pi$ and the logarithms of algebraic numbers.
Some partial linear independence results for odd zeta values by Ball and Rivoal \cite{BR2001,Riv2000} and by the third author \cite{Zud2001} consolidated the use of hypergeometric functions in constructing rational approximations to $L$-values,
while the recent irrationality advancement \cite{CDT2024+} of Calegari, Dimitrov and Tang sets up promising perspectives for this field.
Their proof of $L(2,\chi_{-3})\notin\mathbb Q$ in \cite{CDT2024+}, the irrationality for the $L$-value attached to the Dirichlet modulo~3 character, extends existing arithmetic techniques to new horizons; at the same time it builds on essential ingredients of Ap\'ery's proof and its later re-interpretations.
The use in \cite{CDT2024+} of novel methodology of the so-called holonomicity criteria leads to further arithmetic advances including the proof of the unbounded denominators conjecture \cite{CDT2025} and of the irrationality for products of two logarithms; it also allows Calegari, Dimitrov and Tang \cite{CDT2020+,CDT2025+} to demonstrate the irrationality of $\zeta_2(5)$, the $2$-adic analogue of $\zeta(5)$,\,---\,a result whose Archimedean counterpart is non-existent. 

Our principal goal in this note is to give a new proof of the irrationality of $\zeta_2(5)$, much in a spirit of Ap\'ery's original proof for $\zeta(3)$, via an explicit construction of rational approximations to the 2-adic zeta value. These approximations allow us to also estimate the quality of general rational approximations to the number, that is, to give an upper bound on the irrationality measure $\mu(\zeta_2(5))$\,---\,below we carefully define all the objects involved.
A remarkable feature of our construction is that our approximations satisfy a three-term Ap\'ery-like recurrence relation
\begin{equation}
(n+1)^5\rho_{n+1} - 32(2n+1)(8n^4+16n^3+20n^2+12n+3)\rho_n + 2^{16}n^5\rho_{n-1} = 0
\label{eq:rec}
\end{equation}
for $n=1,2,\dots$\,. Apparently, this recursion did not show up in the literature before; it admits a solution $\{\rho_0,\rho_1,\rho_2,\dots\}=\{1,96,14944,\dots\}$ which is integer-valued\,---\,we give an explicit binomial expression for the latter.
Notice that we do not expect the existence of a three-term Ap\'ery-like recursion for the Archimedean zeta value $\zeta(5)$, though four-term recursions of this type are known~\cite{BZ2022+}.

To summarise, our main result is as follows.

\begin{theorem}\label{main}
The $2$-adic zeta value $\zeta_2(5)$ is irrational. Moreover, we have the following upper bound for its irrationality measure: 
\[  
\mu(\zeta_2(5)) \le \frac{16\log 2}{8\log 2 - 5} = 20.342651\dotsc. 
\]
\end{theorem}

\section{Preliminaries}

\subsection{Irrationality measure}

We first recall the definition of irrationality measure for real numbers.

\begin{definition}
Let $\xi \in \mathbb{R}$. 
Define the \emph{irrationality measure} of $\xi$\,---\,denoted by $\mu(\xi)$\,---\,to be the supremum of the set of real numbers $\mu$ such that 
\[
0< \left| \xi - \frac{A}{B} \right| < \frac{1}{\max\{ |A|, |B| \}^{\mu}}
\]
is satisfied by infinitely many pairs $(A,B) \in \mathbb{Z} \times \mathbb{Z}_{>0}$. 
Note that the supremum can be infinite; in such a case the number $\xi$ is called a Liouville number.
\end{definition}

A classical application of the Borel--Cantelli lemma shows that $\mu(\xi) = 2$ for almost all real numbers $\xi$ in the sense of the Lebesgue measure. 
A deep result of Roth \cite{Rot1955} states that $\mu(\xi)=2$ for any irrational algebraic real number $\xi$. 
It is proved by Euler that $\mu(e)=2$. 
The exact values of $\mu(\pi)$ and $\mu(\zeta(3))$ are also expected to be~2, but we only know that 
\[ 
\mu(\pi) \le 7.103205\dots
\]
from Zeilberger and the third author \cite{ZZ2020}, and
\[
\mu(\zeta(3)) \le 5.513890\dots
\]
from Rhin and Viola \cite{RV2001}.

A similar definition is available for the irrationality measure of a $p$-adic number.

\begin{definition}
Let $p$ be a prime and $\xi \in \mathbb{Q}_p$. 
Define the \emph{irrationality measure} of $\xi$\,---\,again denoted by $\mu(\xi)$\,---\,to be the supremum of the set of real numbers $\mu$ such that 
\[
0< \left| \xi - \frac{A}{B} \right|_p < \frac{1}{\max\{ |A|, |B| \}^{\mu}}
\]
is satisfied by infinitely many pairs $(A,B) \in \mathbb{Z} \times \mathbb{Z}_{>0}$.
\end{definition}

Although not explicitly stated in the literature, the upper bounds
\[
\mu(\zeta_2(3)) \le \frac{12\log 2}{6 \log 2 - 3} = 7.177398\dots
\]
and
\[
\mu(\zeta_3(3)) \le \frac{6\log 3}{3\log 3 -3} = 22.281447\dots
\]
can be extracted from Calegari's work \cite{Cal2005} with the help of a lemma of Bel \cite{Bel2019} stated below, alternatively, from Beukers' result \cite[Theorem 11.2]{Beu2008}.

\begin{lemma}[{Bel \cite[Lemme 3.2]{Bel2019}}]
\label{lem_Bel}
Let $p$ be a prime, $\xi \in \mathbb{Q}_p$ and $\alpha, \beta$ two real numbers satisfying $\alpha > \beta > 0$. 
Suppose that there exist two sequences $(a_n)_{n \ge 1} \subset \mathbb{Z}$ and $(b_n)_{n \ge 1} \subset \mathbb{Z}$ such that
\begin{itemize}
	\item   $|a_n + b_n \xi|_p \le \exp(-\alpha n +o(n))$ as $n \to \infty$\textup;
	\item   $\max\{ |a_n|, |b_n| \} \le \exp(\beta n + o(n))$ as $n \to \infty$\textup;
	\item  $a_nb_{n+1} - a_{n+1}b_n \ne 0$ for any $n \in \mathbb{Z}_{>0}$.
\end{itemize}
Then the $p$-adic number $\xi$ is irrational, with the following estimate for its irrationality measure:
\[ 
\mu(\xi) \le \frac{\alpha}{\alpha - \beta}. 
\]
\end{lemma}

Bel's lemma is a tool from our arsenal for proving Theorem~\ref{main}.

\subsection{Volkenborn integrals}

Let $p$ be a prime number. 
In this subsection, we will recall the definition of Volkenborn integral \cite{Vol1972} and its basic properties.

A function $f\colon \mathbb{Z}_p \to \mathbb{Q}_p$ is said to be \emph{Volkenborn integrable} if the sequence
\[
\frac{1}{p^n} \sum_{k=0}^{p^n-1} f(k)
\]
converges $p$-adically as $n \to \infty$. 
In this case, the value
\[ 
\int_{\mathbb{Z}_p} f(t)\,\d t := \lim_{n \to \infty} \frac{1}{p^n} \sum_{k=0}^{p^n-1} f(k) \in \mathbb{Q}_p 
\]
is called the \emph{Volkenborn integral} of $f$. 

Let $K$ be either $\mathbb{Q}_p$ or $\mathbb{Z}_p$. 
A function $f\colon\mathbb{Z}_p \to K$ is said to be \emph{strictly differentiable} on $\mathbb{Z}_p$\,---\,denoted by $f \in S^{1}(\mathbb{Z}_p,K)$\,---\,if 
\[
f(x) - f(y) = (x-y) g(x,y)
\]
for some continuous function $g(x,y)$ on $\mathbb{Z}_p \times \mathbb{Z}_p$.
It is known that every $f \in S^{1}(\mathbb{Z}_p,\mathbb{Q}_p)$ is Volkenborn integrable (see \cite[p.~264]{Rob2000}).
For our purposes, we note that if a rational function $f(t) \in \mathbb{Q}_p(t)$ has no pole in $\mathbb{Z}_p$, then $f \in S^{1}(\mathbb{Z}_p,\mathbb{Q}_p)$.

The Volkenborn integral has the following behaviour under translations.

\begin{lemma}[{\cite[Proposition 2, p.~265]{Rob2000}}]
\label{lem_translation_formula}
Let $f \in S^{1}(\mathbb{Z}_p,\mathbb{Q}_p)$. 
Then, for any $k \in \mathbb{Z}_{>0}$, we have
\[ 
\int_{\mathbb{Z}_p} f(t+k)\,\d t = \int_{\mathbb{Z}_p} f(t)\,\d t + \sum_{\ell=0}^{k-1} f'(\ell). 
\]
\end{lemma}

For any positive integer $k$, we denote by $k_{-}$ the non-negative integer obtained by deleting the leading $p$-adic digit of~$k$. 
In other words, if the $p$-adic expansion of $k$ assumes the form
$k=a_0+a_1p+\dots+a_{l-1}p^{l-1}+a_{l}p^{l}$ with $a_l \ne 0$, then 
\[
k_{-} = a_0+a_1p+\dots+a_{l-1}p^{l-1}.
\] 
For estimating the $p$-adic norm of the Volkenborn integral of a function $f \in S^{1}(\mathbb{Z}_p,\mathbb{Q}_p)$, we will use the characteristic
\[
\triangle(f) := \min\left\{\inf_{k \ge 1} v_p\left( \frac{f(k)-f(k_{-})}{k-k_{-}} \right),\, 1+v_p(f(0)) \right\} \in \mathbb{Z} \cup \{+\infty\},
\]
where we take the convention that $v_p(0)=+\infty$.
This was first introduced by the second author in \cite{Spr2020} and further developed by the first author in \cite{Lai2025}. 

We have the following properties of $\triangle$ viewed as an operator on the space $S^{1}(\mathbb{Z}_p,\mathbb{Q}_p)$. 

\begin{lemma}[{\cite[Lemma 2.4]{Lai2025}}]
\label{lem_estimation_of_Volkenborn_integral}
Let $f \in S^{1}(\mathbb{Z}_p,\mathbb{Q}_p)$. Then we have
\begin{equation*}
v_p\left( \int_{\mathbb{Z}_p} f(t)\,\d t \right) \ge \triangle(f) - 1.
\end{equation*}
\end{lemma}

\begin{lemma}
\label{lem_properties_of_the_triangle_operator}
The following properties hold for the operator $\triangle$.
\begin{enumerate}
\item[\textup{(a)}] For any $f,g \in S^{1}(\mathbb{Z}_{p},\mathbb{Q}_{p})$, we have
\[
\triangle(f + g) \ge \min\{ \triangle(f), \triangle(g) \}.
\]

\item[\textup{(b)}] For any $f \in S^{1}(\mathbb{Z}_{p},\mathbb{Q}_{p})$ and $C \in \mathbb{Q}_p$, we have
\[
\triangle(C \cdot f) = \triangle(f) + v_p(C).
\]

\item[\textup{(c)}] If $f,g \in S^{1}(\mathbb{Z}_{p},\mathbb{Z}_{p})$, then 
\[ 
\triangle(f\cdot g) \ge \min\{\triangle(f),\triangle(g)\}. 
\]

\item[\textup{(d)}] If $f(t) = \sum_{j=0}^{\infty} a_jt^{j} \in \mathbb{Z}_{p}\llbracket t \rrbracket$ and $\lim_{j \to \infty} |a_j|_p = 0$, then 
\[
\triangle(f) \ge 0.
\]

\item[\textup{(e)}] For $n,j \in \mathbb{Z}$ with $n>0$ and 
\[ 
f(t) = \binom{t+j}{n} = \frac{(t+j)(t+j-1)\dotsb(t+j-n+1)}{n!}, 
\] 
we have 
\[ 
\triangle(f) \ge -\left\lfloor \frac{\log n}{\log p} \right\rfloor.
\] 
\end{enumerate} 
\end{lemma}

\begin{proof}
For parts (c), (d) and (e), see \cite[Lemma 2.5]{Lai2025}; parts (a) and (b) are clear from definition.
\end{proof}

\subsection{$p$-adic zeta values}

%
%
%
We first recall some basic facts about $p$-adic $L$-functions and $p$-adic zeta values. For a prime number $p$, define
\[
	q_p:=\begin{cases}
		p &\text{if}\; p\ne 2,\\
		4 &\text{if}\; p=2.
	\end{cases}
\]
Then $\mathbb{Q}_p^\times \cong p^{\mathbb{Z}} \cdot \mu_{\varphi(q_p)}(\mathbb{Z}_p)\times (1+q_p\mathbb{Z}_p)$ and we write 
\begin{align*}
	\omega &\colon \mathbb{Q}_p^\times\to p^{\mathbb{Z}} \cdot \mu_{\varphi(q_p)}(\mathbb{Z}_p)\subseteq \mathbb{Q}_p^\times, \\ 
	\langle \,\cdot\, \rangle &\colon \mathbb{Q}_p^\times\to (1+q_p\mathbb{Z}_p) \subseteq\mathbb{Q}_p^\times
\end{align*}
for the induced projections; $\omega$ is known as the \emph{Teichm\"uller character}.
For $s\in \mathbb{C}_p\setminus \{1\}$ with $|s|_p<q_p p^{-1/(p-1)}$ and $x\in \mathbb{Q}_p$ with $|x|_p\ge q_p$, we define the \emph{$p$-adic Hurwitz zeta function} by the Volkenborn integral
\[
	\zeta_p(s,x):=\frac{1}{s-1}\int_{\mathbb{Z}_p} \langle t+x \rangle^{1-s}\,\d t.
\]
The \emph{Kubota--Leopoldt $p$-adic $L$-function} associated to a Dirichlet character $\chi$ of conductor $f$ is now defined as follows: Let $M$ be a common multiple of $f$ and $q_p$; then
\[
	L_p(s,\chi):=\frac{\langle M\rangle^{1-s}}{M}\sum_{\substack{j=0\\ p\nmid j}}^{M-1} \chi(j)\zeta_p\bigg( s,\frac{j}{M} \bigg).
\]
It is not difficult to check that this definition does not depend on the choice of~$M$.
For more details, we refer the reader to \cite[Chap.~11]{Coh2007}. 
Finally, we define $p$-adic zeta values as follows.

\begin{definition}
\label{def zeta_p}
For any integer $s \ge 2$, the \emph{$p$-adic zeta value} $\zeta_p(s)$ is given by
\[ 
\zeta_p(s) := L_p(s,\omega^{1-s}).
\]
\end{definition}

Notice that the definitions of $\zeta_p(s)$ by Coleman \cite{Col1984} or Furusho \cite{Fur2004} differ from ours by the Euler factor $(1-p^{-s})^{-1}$ at~$p$. 
However, for each fixed integer $s\ge2$, this factor is a non-zero rational number, so it does not matter which definition is used for proving irrationality. 
Let us also note that Beukers in \cite{Beu2008} refers to $L_p(s,\chi_0)$ corresponding to the principal character $\chi=\chi_0$ as to $\zeta_p(s)$.
The name of `$p$-adic zeta' is justified by the following characterisation (see \cite[Lemma 2.4]{Cal2005}):
\[
\zeta_p(s) = \lim_{\substack{k\to s \;\text{$p$-adically}\\k \in \mathbb{Z}_{<0}, \; k\equiv s \pmod{p-1}}} \zeta(k) \in \mathbb{Q}_p.
\]
In other words, the $p$-adic zeta value $\zeta_p(s)$ is a $p$-adic limit of special values of the Riemann zeta function at negative integers; in particular, we have $\zeta_p(s)=0$ for any even positive integer $s$.

In the special case $p=2$ and $s \in \mathbb{Z}_{\ge 2}$, it follows from Definition \ref{def zeta_p} that
\begin{align}
\zeta_2(s)
&= L_2(s,\omega^{1-s}) \notag\\ 
&= \frac{1}{4}\left( \omega(1)^{1-s}\zeta_2\left( s,\frac{1}{4}\right) + \omega(3)^{1-s}\zeta_2\left( s,\frac{3}{4}\right) \right) \notag\\
&= \frac{1}{4}\left( \frac{1}{s-1}\int_{\mathbb{Z}_2} \frac{\d t}{\langle t+1/4 \rangle^{s-1}} + \frac{(-1)^{1-s}}{s-1}\int_{\mathbb{Z}_2} \frac{\d t}{\langle t+3/4 \rangle^{s-1}} \right) \notag\\
&= \frac{1}{4(s-1)} \left( \int_{\mathbb{Z}_2} \frac{\d t}{(1+4t)^{s-1}} +  \int_{\mathbb{Z}_2} \frac{\d t}{(3+4t)^{s-1}}  \right) \notag\\
&= \frac{1}{4(s-1)} \bigg( \lim_{N \to \infty}  \frac{1}{2^N} \sum_{k=0}^{2^N -1} \frac{1}{(1+4k)^{s-1}} + \lim_{N \to \infty}  \frac{1}{2^N} \sum_{k=0}^{2^N -1} \frac{1}{(3+4k)^{s-1}}\bigg).
\label{271}
\end{align}
We can further translate this expression into a Volkenborn integral.

\begin{lemma}
\label{lem_shift_by_1/2}
For any $s \in \mathbb{Z}_{\ge 2}$ we have
\[  
\frac{1}{s-1}\int_{\mathbb{Z}_2} \frac{\d t}{(t+1/2)^{s-1}} = 2^{s} \cdot \zeta_2(s).
 \]
\end{lemma}

\begin{proof}
We have
\begin{align}
\frac{1}{s-1}\int_{\mathbb{Z}_2} \frac{\d t}{(t+1/2)^{s-1}}
&= \frac{2^{s-1}}{s-1} \int_{\mathbb{Z}_2} \frac{\d t}{(1+2t)^{s-1}} \notag\\
&= \frac{2^{s-1}}{s-1} \lim_{N \to \infty} \frac{1}{2^{N+1}} \sum_{k=0}^{2^{N+1}-1} \frac{1}{(1+2k)^{s-1}} \notag\\
&= \frac{2^{s-1}}{s-1} \lim_{N \to \infty} \frac{1}{2^{N+1}}\bigg( \sum_{k=0}^{2^N-1} \frac{1}{(1+4k)^{s-1}} + \sum_{k=0}^{2^N-1} \frac{1}{(3+4k)^{s-1}} \bigg).
\label{272}
\end{align}
Comparing \eqref{272} with \eqref{271}, we obtain the result.
\end{proof}

\section{Rational functions and linear forms}

In this section, we first introduce a sequence of rational functions. 
Then we make use of them to construct linear forms in $1$ and $\zeta_2(5)$. 
Recall that the Pochhammer symbol $(t)_n$ is defined by $(t)_{n}:=t(t+1)\cdots(t+n-1)$ for $n \in \mathbb{Z}_{>0}$ with the convention $(t)_0:= 1$.

\begin{definition}
\label{def_R_n(t)}
For $n \in \mathbb{Z}_{\ge0}$, define the rational function $R_n(t) \in \mathbb{Q}(t)$ by
\[ 
R_n(t) :=  2^{8n} \cdot (2t+n) \cdot \frac{(t+1/2)_n^4}{(t)_{n+1}^4}. 
\]
\end{definition}  

We mention that similar choices of rational function appear in \cite{Bel2019,Lai2025,Riv2017+}. 

For a rational function $R(t)=P(t)/Q(t)$, where $P$, $Q$ are polynomials, we define the degree of $R(t)$ by $\deg R := \deg P - \deg Q$. 
Note that for any $n \in \mathbb{Z}_{\ge0}$ we have 
\begin{equation}\label{deg_R=-3}
\deg R_n = -3. 
\end{equation}
In particular, our rational function admits a partial-fraction decomposition
\begin{equation}
\label{def_rik}
R_n(t) =: \sum_{i=1}^{4}\sum_{k=0}^{n} \frac{r_{n,i,k}}{(t+k)^i},
\end{equation}
with the coefficients $r_{n,i,k} \in \mathbb{Q}$ uniquely determined by $R_n(t)$. 

\begin{definition}
\label{def_Sn}
For $n \in \mathbb{Z}_{\ge0}$, define the following $2$-adic quantity:
\[ 
S_n :=  - \int_{\mathbb{Z}_2} R_n'\Big( t+\frac{1}{2} \Big)\,\d t, 
\]
where $R_n'(t)$ is the derivative function of $R_n(t)$ with respect to $t$.
\end{definition}

We claim that $S_n$ is a linear form in $1$ and $\zeta_2(5)$ with rational coefficients.

\begin{lemma}
\label{lem_linear_forms}
For any $n \in \mathbb{Z}_{\ge0}$, we have
\[ 
S_n = \rho_{n,0} + \rho_{n,3} \cdot \zeta_2(5), 
\]
where
\begin{align}
\rho_{n,0} &= -\sum_{i=1}^{4}\sum_{k = 0}^{n}\sum_{\ell=1}^k \frac{i(i+1)r_{n,i,k}}{(\ell-1/2)^{i+2}} \in \mathbb{Q},
\label{def_rho_0}\\
\rho_{n,3} &= 384 \sum_{k=0}^{n} r_{n,3,k} \in \mathbb{Q}.
\label{def_rho_3}
\end{align}
The convention here and in what follows is that the empty sum \textup(when $k=0$\textup) is understood as $0$.
\end{lemma}

\begin{proof}
By Definition \ref{def_Sn} and Equation \eqref{def_rik}, we have
\begin{equation}
\label{341}
S_n =\sum_{i=1}^{4} \sum_{k=0}^{n} ir_{n,i,k}  \int_{\mathbb{Z}_2} \frac{\d t}{(t+k+1/2)^{i+1}}.
\end{equation}
By Lemma \ref{lem_translation_formula} and Lemma \ref{lem_shift_by_1/2}, we have
\begin{align}
\int_{\mathbb{Z}_2} \frac{\d t}{(t+k+1/2)^{i+1}}
&= \int_{\mathbb{Z}_2} \frac{\d t}{(t+1/2)^{i+1}} -(i+1)\sum_{\ell=1}^k \frac{1}{(\ell-1/2)^{i+2}} \notag\\
&= (i+1)2^{i+2}\zeta_2(i+2) -(i+1)\sum_{\ell=1}^k \frac{1}{(\ell-1/2)^{i+2}}.
\label{342}
\end{align}
Substituting \eqref{342} into \eqref{341}, and using $\zeta_2(s)=0$ for even $s$, we obtain 
\[ 
S_n = \rho_{n,0} + \left(16\sum_{k = 0}^{n} r_{n,1,k}\right) \cdot \zeta_2(3) + \rho_{n,3} \cdot \zeta_2(5). 
\]
By \eqref{def_rik} and \eqref{deg_R=-3}, we have 
\[ 
\sum_{k=0}^{n} r_{n,1,k} = \lim_{t \to \infty} tR_n(t) = 0. 
\]
Therefore, $S_n = \rho_{n,0} + \rho_{n,3} \cdot \zeta_2(5)$, as desired.
\end{proof}

\begin{remark}
\label{rem_diff_arch}
Our derivation `secretly' corresponds to the equality
\[ 
\sum_{m=0}^{\infty} R_n''\Big(m+\frac{1}{2}\Big) = \rho_{n,0} + \rho_{n,3} \cdot \left(1 - 2^{-5} \right)\cdot \zeta(5)
\]
in the Archimedean case.
The coefficients $\rho_{n,0} \in \mathbb{Q}$ and $\rho_{n,3} \in \mathbb{Q}$ in this formula are exactly the same as in Lemma~\ref{lem_linear_forms}, and $\left(1 - 2^{-5} \right)\cdot \zeta(5)$ is the value of the Riemann zeta function at $s=5$ with the Euler factor at $2$ removed. 
This correspondence to linear forms in classical zeta values explains the motivation behind the artificial minus sign appearing in Definition \ref{def_Sn}.
\end{remark}

\section{A three-term recursion for $\rho_{n,0}$ and $\rho_{n,3}$}

With the help of Zeilberger’s algorithm of creative telescoping \cite[Chap.~6]{PWZ1997}, we obtain the following recursion for $(R_n(t))_{n \ge 0}$.

\begin{lemma}
For $n \in \mathbb{Z}_{>0}$, let
\begin{align}
T_n(t)
&=\big(8(2n+1)t^4 + 48n(2n+1)t^3 + 2(2n+1)(48n^2-6n-5)t^2
\notag\\ &\quad
+ 2(80n^4 + 16n^3 - 28n^2 - 3n + 3)t
\notag\\ &\quad
+ (48n^5 - 24n^3 + 3n^2 + 4n - 1)\big)
\cdot\frac{2^{8n+4}(t-1/2)_n^4}{(t)_{n+1}^4}.
\label{def_T_n(t)}
\end{align}
Then
\begin{multline}
(n+1)^5 R_{n+1}(t) - 32(2n+1)(8n^4+16n^3+20n^2+12n+3) R_n(t) + 2^{16}n^5 R_{n-1}(t)
\\
= T_n(t+1) - T_n(t).
\label{001}
\end{multline} 
\end{lemma}

\begin{proof}
Dividing both sides of \eqref{001} by $2^{8(n+1)}(t+1/2)_{n}^4 / (t)_{n+1}^4$ and clearing the denominators in the result, reduce verification of Identity \eqref{001} to a linear-algebra check of identity between two polynomials of degree at most 12 in~$t$ and at most 14 in~$n$.
\end{proof}

The telescoping identity for $(R_n(t))_{n \ge 0}$ induces the following recursive formulae for both $(\rho_{n,0})_{n \ge 0}$ and $(\rho_{n,3})_{n \ge 0}$.

\begin{lemma}
\label{lem_rec}
\begin{enumerate}
\item[\textup{(a)}] For each $i \in \{0,3\}$, the sequence $(\rho_{n,i})_{n \ge 0}$ satisfies the three-term relation \eqref{eq:rec}.
\item[\textup{(b)}] For any $n \in \mathbb{Z}_{\ge0}$, we have the `determinant' formula
\[
\rho_{n,0}\rho_{n+1,3} - \rho_{n+1,0}\rho_{n,3} = \frac{3 \cdot2^{16n+18}}{(n+1)^5} \ne 0. 
\]
\end{enumerate}
\end{lemma}

\begin{proof}
We first prove part (a). 
By \eqref{def_T_n(t)}, the function $T_n(t)$ is a rational function in $t$ with $\deg T_n = 0$, hence its partial-fraction decomposition assumes the form
\[
T_n(t) =: c_{n} + \sum_{i=1}^{4}\sum_{k=0}^n \frac{a_{n,i,k}}{(t+k)^i}
\]
for some $c_n \in \mathbb{Q}$ and $a_{n,i,k} \in \mathbb{Q}$.

By Equations \eqref{001}, \eqref{def_rik} and the uniqueness of partial-fraction decomposition, we obtain for $n \in \mathbb{Z}_{>0}$, any $k \in \{0,1,\dots,n+1\}$ and $i \in\{1,2,3,4\}$ that
\begin{equation}
(n+1)^5 r_{n+1,i,k} - 32(2n+1)(8n^4+16n^3+20n^2+12n+3) r_{n,i,k} + 2^{16}n^5 r_{n-1,i,k} = a_{n,i,k-1} - a_{n,i,k},
\label{421}
\end{equation}
where all the coefficients outside the eligible range vanish:
$a_{n,i,-1} = a_{n,i,n+1} =0$ and $r_{n,i,n+1} =r_{n-1,i,n} =r_{n-1,i,n+1} = 0$.
Taking $i=3$ in \eqref{421}, summing \eqref{421} over $k \in \{0,1,\dots,n+1\}$ and using \eqref{def_rho_3}, we obtain
\begin{equation}
\label{rec_for_rho_3}
(n+1)^5\rho_{n+1,3} - 32(2n+1)(8n^4+16n^3+20n^2+12n+3)\rho_{n,3} + 2^{16}n^5\rho_{n-1,3} = 0.
\end{equation}
In a different direction, using Definition \ref{def_Sn}, Equation \eqref{001}, Lemma \ref{lem_translation_formula} and Equation \eqref{def_T_n(t)}, we have
\begin{align}
&
(n+1)^5 S_{n+1} - 32(2n+1)(8n^4+16n^3+20n^2+12n+3) S_n + 2^{16}n^5 S_{n-1}
\notag\\
&\; =-\int_{\mathbb{Z}_2} \Big( (n+1)^5 R_{n+1}'\Big(t +\frac{1}{2} \Big) - 32(2n+1)(8n^4+16n^3+20n^2+12n+3) R_n'\Big(t +\frac{1}{2} \Big)
\notag\\ &\;\quad
+ 2^{16}n^5 R_{n-1}'\Big(t +\frac{1}{2} \Big)  \Big) \,\d t \notag\\
&\; = -\int_{\mathbb{Z}_2} \Big( T_n'\Big(t +\frac{3}{2} \Big) - T_n'\Big(t +\frac{1}{2} \Big) \Big) \,\d t
= -T_n''\Big( \frac{1}{2} \Big) 
= 0.
\label{rec_for_S_n}
\end{align}
By Lemma \ref{lem_linear_forms}, we have $\rho_{n,0} = S_n - \rho_{n,3}\cdot\zeta_2(5)$. 
Therefore, Equations \eqref{rec_for_rho_3} and \eqref{rec_for_S_n} imply that the sequence $(\rho_{n,0})_{n \ge 0}$ satisfies the recursion \eqref{eq:rec} as well.
This completes the proof of part~(a).

Now we prove part (b). 
For $n=0$, the determinant formula is checked by a straightforward computation: since $\rho_{0,0}=0$, $\rho_{0,3}=768$, $\rho_{1,0}=-1024$, $\rho_{1,3}=73728$, we deduce that
\[
\rho_{0,0}\rho_{1,3} - \rho_{1,0}\rho_{0,3} = 3 \cdot 2^{18}.
\]
For $n\ge 1$ it follows inductively by using the recursions from part~(a) that
\[
\rho_{n,0}\rho_{n+1,3} - \rho_{n+1,0}\rho_{n,3}
=2^{16} \frac{n^5}{(n+1)^5}(\rho_{n-1,0}\rho_{n,3} - \rho_{n,0}\rho_{n-1,3}).
\qedhere
\]
\end{proof}

\section{Arithmetic properties}
In this section, we will investigate the arithmetic properties of the coefficients $\rho_{n,i}$ of linear forms $S_n = \rho_{n,0} + \rho_{n,3} \cdot \zeta_2(5)$ as defined in Lemma \ref{lem_linear_forms}.

In what follows, $d_n$ stands for the least common multiple of $1,2,\dots,n$ where $n \in \mathbb{Z}_{>0}$. 

Note that building blocks of the rational function $R_n(t)$ are rational functions $2t+n$,
\[
F(t) = 2^{2n} \cdot \frac{(t+1/2)_n}{n!}
\quad\text{and}\quad
G(t) = \frac{n!}{(t)_{n+1}}.
\]
Applying a somewhat standard argument based on their properties (see \cite[Lemma 4.2]{Lai2025+b} and \cite[Lemma 16]{Zud2004}), we have $d_n^{4-i} \cdot r_{n,i,k} \in \mathbb{Z}$ for any $n \in \mathbb{Z}_{>0}$, $1 \le i \le 4$ and $0 \le k \le n$. 
Then, adapting the strategy of the proof of \cite[Lemma 2]{FSZ2019} (as in \cite[Lemma 4.6]{Lai2025}), it is elementary to establish the following inclusions.

\begin{lemma}
\label{lem_arith_rho}
For $n \in \mathbb{Z}_{>0}$,
\[ 
d_{n} \cdot \rho_{n,3} \in \mathbb{Z}
\quad\text{and}\quad
d_{n}^6 \cdot \rho_{n,0} \in \mathbb{Z}. 
\]
\end{lemma}

The result of the lemma is weaker than the expectation
\begin{equation}
\rho_{n,3} \in \mathbb{Z}
\quad\text{and}\quad
d_{n}^5 \cdot \rho_{n,0} \in \mathbb{Z}
\quad\text{for all}\; n\in\mathbb Z_{>0},
\label{den-con}
\end{equation}
based on a numerical check for $n\le 5000$ (using the recurrence equation \eqref{eq:rec}) and on a `usual' behaviour of the coefficients of related linear forms in (Archimedean) zeta values.
Such expectations are dubbed `denominator conjectures' in this context, and many of them were systematically established by Krattenthaler and Rivoal in \cite{KR2007}.
Modifying their methodology, while still making use of the famous Andrews transformation \cite{And1975} of terminating very-well-poised hypergeometric series into a multiple hypergeometric sum (as $q\to1$), we demonstrate below that indeed $\rho_{n,3} \in \mathbb{Z}$ for $n\in\mathbb Z_{>0}$ and a slightly weaker version of the companion inclusions in~\eqref{den-con}.
The latter weakness does not affect the asymptotic behaviour of the denominators of $\rho_{n,0}$ as $n\to\infty$, so that our arithmetic findings are as good as required when it comes to their application for Theorem~\ref{main}.

\begin{lemma}
\label{lem:rho0}
The solution $(\rho_n)_{n\ge0}$ of the difference equation \eqref{eq:rec} with the initial conditions $\rho_0=1$, $\rho_1=96$ is given by the binomial double sum
\[
\rho_n=\sum_{0\le i\le k\le n}2^{4(n-k)}{\binom{2i}i}^2\binom{2n-2i}{n-i}\binom{2k-2i}{k-i}{\binom{2k}k}^2\binom{2n-2k}{n-k}.
\]
Furthermore, for the coefficients $\rho_{n,3}$ given in \eqref{def_rho_3} we have $\rho_{n,3}=768\rho_n\in\mathbb Z$ for $n\in\mathbb Z_{\ge0}$.
\end{lemma}

\begin{proof}
Though one can easily verify that the double sum indeed satisfies the recursion \eqref{eq:rec} with the help of multi-sum algorithms of creative telescoping, we give a human proof below, which also indicates how we arrived at the explicit expression for $\rho_n$.

We define our sequence $(\rho_n)_{n\ge0}$ via the formula
\[
\rho_n:=\frac12\sum_{k=0}^nr_{n,3,k}=\frac12\sum_{k\in\mathbb Z}r_{n,3,k},
\]
where the coefficients $r_{n,3,k}$ in the partial-fraction decomposition \eqref{def_rik} are given by
\[
r_{n,3,k}=\frac{\d}{\d t}\big(R_n(t)(t+k)^4\big)\Big|_{t=-k} \quad\text{for}\; k=0,1,\dots,n.
\]
The inspection of this formula for $n=0$ and $1$ reveals that $\rho_0=1$ and $\rho_1=96$, while the recursion \eqref{eq:rec} for the sequence follows from $\rho_n=\rho_{n,3}/768$ (as shown in Lemma~\ref{lem_linear_forms}) and Lemma~\ref{lem_rec}.
Now consider the $\eps$-deformation
\[
R_n(t;\eps)=2^{8n}(2t+n+\eps)\,\frac{(t+1/2)_n^2(t+\eps+1/2)_n^2}{(t)_{n+1}^2(t+\eps)_{n+1}^2}
\]
of the function $R_n(t)$ with the motive that
\[
\frac{\d}{\d t}\big(R_n(t)(t+k)^4\big)
=2\frac{\p}{\p\eps}\big(R_n(t;\eps)(t+k)^2(t+\eps+k)^2\big)\bigg|_{\eps=0},
\]
so that
\begin{align*}
\rho_n
&=\frac12\sum_{k=0}^n\frac{\d}{\d t}\big(R_n(t)(t+k)^4\big)\bigg|_{t=-k}
\\
&=\sum_{k=0}^n\frac{\p}{\p\eps}\big(R_n(t;\eps)(t+k)^2(t+\eps+k)^2\big)\bigg|_{\eps=0,\;t=-k}
\\
&=\sum_{k=0}^n\frac{\p}{\p\eps}\big(R_n(t-k;\eps)t^2(t+\eps)^2\big)\bigg|_{\eps=0,\;t=0}
\\
&=2^{8n}\frac{\p}{\p\eps}\bigg(\sum_{k=0}^n
(n+\eps-2k)\,\frac{(1/2-k)_n^2(\eps+1/2-k)_n^2}{k!^2(n-k)!^2(1-\eps)_k^2(1+\eps)_{n-k}^2}
\bigg)\bigg|_{\eps=0}
\\
&=2^{8n}\frac{\p}{\p\eps}\bigg((n+\eps)\frac{(\frac12)_n^2(\frac12+\eps)_n^2}{n!^2(1+\eps)_n^2}
\cdot
{}_9V_8(-n-\eps; \, -n-\eps, \, \tfrac12, \, \tfrac12, \, \tfrac12-\eps, \, \tfrac12-\eps, \, -n, \, -n)\bigg)\bigg|_{\eps=0},
\end{align*}
where the notation
\begin{equation}\label{eq:def-very-well-poised}
{}_{m+2}V_{m+1}(a_0;a_1,\dots,a_m)
=\sum_{k=0}^\infty\frac{(a_0+2k)\prod_{j=0}^m(a_j)_k}{a_0\,k!\prod_{j=1}^m(1+a_0-a_j)_k}
\end{equation}
is used for the very-well-poised hypergeometric series evaluated at~$1$.

Taking $m=2$ and $a=-n-\eps$, $b_1=\frac12$, $c_1=-n-\eps$, $b_2=\frac12-\eps$, $c_2=-n$, $b_3=\frac12$, $c_3=\frac12-\eps$ in \cite[Th\'eor\`eme~8]{KR2007} we obtain the following double sum expression:
\begin{align*}
&
(n+\eps)\frac{(\frac12)_n^2(\frac12+\eps)_n^2}{n!^2(1+\eps)_n^2}
\cdot{}_9V_8(-n-\eps; \, -n-\eps, \, \tfrac12, \, \tfrac12, \, \tfrac12-\eps, \, \tfrac12-\eps, \, -n, \, -n)
\\ &\quad
=(n+\eps)\frac{(\frac12)_n^2(\frac12+\eps)_n^2}{n!^2(1+\eps)_n^2}
\cdot\frac{(-\eps)\cdot(1-n-\eps)_{n-1}(-n)_n}{(\frac12-n-\eps)_n(\frac12-n)_n}
\\ &\quad\quad
\times
\sum_{0\le i\le k\le n}\frac{(\frac12)_i(\frac12-\eps)_i(-n)_i}{i!^2(\frac12-n-\eps)_i}
\cdot
\frac{(\frac12)_{k-i}(\frac12)_k(\frac12-\eps)_k(-n)_k}{(k-i)!k!(1-\eps)_k(\frac12-n)_k}
\displaybreak[2]\\ &\quad
=\eps\cdot\frac{(\frac12)_n^4}{n!^2(\frac12-n)_n^2}
\sum_{0\le i\le k\le n}\frac{(\frac12)_i^2(-n)_i}{i!^2(\frac12-n)_i}
\cdot
\frac{(\frac12)_{k-i}(\frac12)_k^2(-n)_k}{(k-i)!k!^2(\frac12-n)_k}
+O(\eps^2)
\displaybreak[2]\\ &\quad
=\eps\sum_{0\le i\le k\le n}\frac{(\frac12)_i^2(\frac12)_{n-i}}{i!^2(n-i)!}
\cdot
\frac{(\frac12)_{k-i}(\frac12)_k^2(\frac12)_{n-k}}{(k-i)!k!^2(n-k)!}
+O(\eps^2)
\quad\text{as}\; \eps\to0.
\end{align*}
This implies that
\[
\rho_n=2^{8n}\sum_{0\le i\le k\le n}\frac{(\frac12)_i^2(\frac12)_{n-i}}{i!^2(n-i)!}
\cdot
\frac{(\frac12)_{k-i}(\frac12)_k^2(\frac12)_{n-k}}{(k-i)!k!^2(n-k)!};
\]
finally, one uses $(\frac12)_i/i!=2^{-2i}\binom{2i}{i}$ multiple times to arrive at the formula for $\rho_n$ claimed.
\end{proof}

We now focus on the sequence $(\rho_{n,0})_{n\ge0}$ defined in \eqref{def_rho_0} and write it as
\begin{equation}
\rho_{n,0} = -\sum_{\ell=1}^n \bigg(\sum_{k=\ell}^{n} \sum_{i=1}^{4} \frac{i(i+1)r_{n,i,k}}{(\ell-\frac12)^{i+2}} \bigg).
\label{def_rho_0-new}
\end{equation}

\begin{lemma}
\label{lem-p-est}
For any prime $p>\max\{\sqrt{2n},3\}$, we have
\[
v_p(\rho_{n,0}) \ge -5.
\]
\end{lemma}

\begin{proof}
Since
\begin{align*}
\frac12\sum_{i=1}^{4} \frac{i(i+1)r_{n,i,k}}{(\ell-\frac12)^{i+2}}
&=-\frac{1}{3!} \, \frac{\d^3}{\d t^3} \bigg( R_n(t)(t+k)^4 \cdot \frac{1}{(t+k-\ell+\frac12)^3} \bigg)\bigg|_{t=-k} \\
&=-\frac{1}{3!} \, \frac{\d^3}{\d\eps^3} \bigg( R_n(-k+\eps)\eps^4 \cdot \frac{1}{(\eps-\ell+\frac12)^3} \bigg)\bigg|_{\eps=0}
\displaybreak[2]\\
&=\frac{1}{3!} \, \frac{\d^3}{\d\eps^3} \bigg( 2^{8n} (n-2k+2\eps)(\ell-\tfrac12-\eps) 
\\ &\qquad\qquad\times
\frac{(\ell+\frac12-\eps)_{k-\ell}^4 (\frac12-\eps)_{\ell-1}^4 (\frac12+\eps)_{n-k}^4}{(1-\eps)_k^4(1+\eps)_{n-k}^4} \bigg)\bigg|_{\eps=0},
\end{align*}
we obtain
\[
\sum_{k=\ell}^{n} \sum_{i=1}^{4} \frac{i(i+1)r_{n,i,k}}{(\ell-\frac12)^{i+2}}
=\frac13\,\frac{\d^3}{\d\eps^3}T_{n,\ell}(\eps)\bigg|_{\eps=0},
\]
where
\[
T_{n,\ell}(\eps):=2^{8n}(\ell-\tfrac12-\eps) (\tfrac12-\eps)_{\ell-1}^4 \sum_{k=\ell}^n (n-2k+2\eps) \frac{(\ell+\frac12-\eps)_{k-\ell}^4  (\frac12+\eps)_{n-k}^4}{(1-\eps)_k^4(1+\eps)_{n-k}^4}.
\]
Using the argument in the proof of \cite[Chap.~9, Corollaire~1]{KR2007} we recognise the sum as a limiting case of very-well-poised hypergeometric series,
\begin{align*}
T_{n,\ell}(\eps)
&= 2^{8n}(\ell-\tfrac12-\eps) (\tfrac12-\eps)_{\ell-1}^4 (n-2\ell+2\eps) \frac{(\frac12+\eps)_{n-\ell}^4}{(1-\eps)_{\ell}^4 (1+\eps)_{n-\ell}^4}
\\ &\quad
\times \lim_{\delta \to 0} {}_{13}V_{12}(a; \, b_1, \, c_1, \, \dots, \, b_5, \, c_5(\delta), \, -N),
\end{align*}
where we use the notation ${}_{13}V_{12}(\dots)$ for the very-well-poised hypergeometric series introduced in \eqref{eq:def-very-well-poised} with the parameters
\begin{gather*}
	N = n-\ell, \quad
	a = -n+2\ell -2\eps, \quad
	b_1=\cdots=b_4 = -n+\ell-\eps, \\
	c_1=\cdots=c_4 = \ell+\tfrac{1}{2}-\eps, \quad
	b_5 = 1, \quad
	c_5(\delta) = \ell+1-2\eps - \delta.
\end{gather*}
Note that the limit $\delta \to 0$ assures that the resulting sum is finite. Applying the Andrews transformation \cite[Th\'eorème 8]{KR2007} leads to the quadruple-sum expression
\begin{align*}
T_{n,\ell}(\eps) =\sum_{0 \le i_1 \le \dots \le i_4 \le n-\ell} F_{i_1,i_2,i_3,i_4}(\eps),
\end{align*}
where
\begin{align*}
F_{i_1,i_2,i_3,i_4}(\eps)
&=2^{2n}  \frac{(\frac12-\eps)_{\ell+i_1}(\frac12+\eps)_{n-\ell-i_1}}{(1-\eps)_{\ell+i_1}(1+\eps)_{n-\ell-i_1}} \\ &\quad
\times2^{2n} \frac{(\ell+\frac12-\eps)_{i_2}(\frac12-\eps)_{\ell-1}(\frac12+\eps)_{n-\ell-i_2}}{(1-\eps)_{\ell+i_2}(1+\eps)_{n-\ell-i_2}} \\ &\quad
\times2^{2n} \frac{(\ell+\frac12-\eps)_{i_3}(\frac12-\eps)_{\ell-1}(\frac12+\eps)_{n-\ell-i_3}}{(1-\eps)_{\ell+i_3}(1+\eps)_{n-\ell-i_3}} \\ &\quad
\times (-n+\ell-1)\binom{2i_1}{i_1} \binom{2i_2-2i_1}{i_2-i_1}\binom{2i_3-2i_2}{i_3-i_2}\binom{2i_4-2i_3}{i_4-i_3}
\\ &\quad
\times \frac{\ell-2\eps}{i_4+1} \cdot \frac{(\ell+1-2\eps)_{i_4}}{(\ell+1-\eps)_{i_4}}
\cdot 2^{2\ell} \frac{(\frac12-\eps)_{\ell-1}}{(1-\eps)_{\ell}}
\\ &\quad
\times 2^{2n-2\ell-2i_4} \frac{(\frac12+\eps)_{n-\ell-i_4}}{(1+\eps)_{n-\ell-i_4}}
\cdot\frac{(n-\ell+1-i_4)_{i_4}}{(n-\ell+1-i_4+\eps)_{i_4}}.
\end{align*}

We now fix a prime $p>\max\{\sqrt{2n},3\}$. Using the formula $2^{2i} \left(\frac{1}{2}\right)_i/i!=\binom{2i}{i}$ repeatedly, observe that
\begin{align*}
F_{i_1,i_2,i_3,i_4}(\eps) |_{\eps = 0}
&=-\frac{16}{(2\ell-1)^2} \, \frac{n-\ell+1}{i_4+1}
\binom{2(\ell+i_1)}{\ell+i_1}\binom{2(n-\ell-i_1)}{n-\ell-i_1}
\\ &\quad
\times \binom{2(\ell+i_2)}{\ell+i_2}\binom{2(n-\ell-i_2)}{n-\ell-i_2}
\binom{2(\ell+i_3)}{\ell+i_3}\binom{2(n-\ell-i_3)}{n-\ell-i_3}
\\ &\quad
\times \binom{2i_1}{i_1}\binom{2(i_2-i_1)}{i_2-i_1}\binom{2(i_3-i_2)}{i_3-i_2} \binom{2(i_4-i_3)}{i_4-i_3}
\\ &\quad
\times \binom{2(\ell-1)}{\ell-1}\binom{2(n-\ell-i_4)}{n-\ell-i_4}.
\end{align*}
Keep in mind that $p>\sqrt{2n}$ implies, for $\ell,i_4\le n$, the estimates $v_p(2\ell -1)\le 1$ and $v_p(i_4+1)\le 1$, so that for trivial reasons we always have
\[
v_p\big(  F_{i_1,i_2,i_3,i_4}(\eps) |_{\eps = 0} \big) \ge -3.
\]

We claim that, for any $i_1,\dots,i_4$ such that $0 \le i_1 \le \dots \le i_4 \le n-\ell$, we have the stronger statement
\begin{equation}
\label{claim}
v_p\big(  F_{i_1,i_2,i_3,i_4}(\eps) |_{\eps = 0} \big) \ge -2.
\end{equation}
If $p \nmid (2\ell-1)$ or $p \nmid (i_4+1)$, then \eqref{claim} clearly holds; therefore, in the sequel we assume that
\[
p \mid (2\ell-1) \quad\text{and}\quad  p \mid (i_4+1).
\] 

Note that for any non-negative integer $m  \le n$, we have
\[
p \mid \binom{2m}{m} \iff (m \bmod p) \ge \frac{p+1}{2},
\]
where $(m \bmod p)$ is the least non-negative residue of $m$ modulo $p$.

If $(i_1 \bmod p) \ge \frac{p+1}{2}$, then we have $p \mid \binom{2i_1}{i_1}$, hence Claim~\eqref{claim} is true. 
If $(i_1 \bmod p) \le \frac{p-3}{2}$, then $p \mid \binom{2(\ell+i_1)}{\ell+i_1}$, validating Claim \eqref{claim} again. 
In the remaining situation, we have
\begin{align*}
i_1 \equiv \frac{p-1}{2} \pmod{p}, \quad
\ell \equiv \frac{p+1}{2} \pmod{p} \quad\text{and}\quad
i_4 \equiv p-1 \pmod{p}.
\end{align*}
Now, if $(n \bmod p) > \frac{p-1}{2}$, then $p \mid \binom{2(n-\ell-i_1)}{n-\ell-i_1}$, so that Claim~\eqref{claim} is true. 
If $(n \bmod p) = \frac{p-1}{2}$, then $p \mid (n-\ell+1)$, again validating the claim. 
Finally, if $(n \bmod p) < \frac{p-1}{2}$, then $p \mid \binom{2(n-\ell-i_4)}{n-\ell-i_4}$, so that Claim~\eqref{claim} holds true. 
This completes the verification of~\eqref{claim}.

Using the induction argument as in \cite[Lemma 17]{Zud2004}, we deduce from \eqref{claim} that
\[
v_p\bigg( \frac{\d^{\lambda}}{\d\eps^{\lambda}} F_{i_1,i_2,i_3,i_4}(\eps) \bigg|_{\eps = 0} \bigg) \ge -2-\lambda \quad \text{for all}\; \lambda \in \mathbb{Z}_{\ge 0}.
\]
Taking $\lambda = 3$ we conclude that the $p$-adic order of
\[
\sum_{k=\ell}^{n}\sum_{i=1}^{4}\frac{i(i+1)r_{n,i,k}}{(\ell -\frac{1}{2})^{i+2}}
=\frac13\sum_{0 \le i_1 \le \dots \le i_4 \le n-\ell} \frac{\d^3}{\d\eps^3}F_{i_1,i_2,i_3,i_4}(\eps)\bigg|_{\eps = 0}
\]
is at least $-5$ for any $\ell\in\{1,2,\dots,n\}$.
Finally, application of Equation \eqref{def_rho_0-new} completes the proof for the $p$-adic order of $\rho_{n,0}$.
\end{proof}

\begin{lemma}
\label{almost_GDC}
For any $n \in \mathbb{Z}_{>0}$, we have
\[ 
\Phi_n^{-1} d_n^6 \cdot \rho_{n,0} \in \mathbb{Z}, 
\]
where $\Phi_n$ denotes the following product over primes:
\[
\Phi_n = \prod_{\max\{\sqrt{2n},3\}<p \le n} p. 
\]
\end{lemma}

\begin{proof}
This follows from Lemmas \ref{lem_arith_rho} and \ref{lem-p-est}.
\end{proof}

Finally, notice consequences of the prime number theorem:
\begin{equation}
d_n = e^{n + o(n)} \quad\text{and}\quad
\Phi_n = e^{n + o(n)} \quad\text{as}\; n \to \infty.
\label{eq:PNT}
\end{equation}
This makes the loss in Lemma~\ref{almost_GDC} vs the expectation \eqref{den-con} asymptotically negligible.

\section{Asymptotic estimates}

\begin{lemma}\label{lem_analysis}
We have
\[ 
\max\{ |\rho_{n,0}|, |\rho_{n,3}| \} \le 2^{8n+o(n)} \quad \text{as}\; n \to \infty.  
\]
\end{lemma}

\begin{proof}
By Lemma~\ref{lem_rec} both sequences of the coefficients satisfy the same difference equation \eqref{eq:rec} whose characteristic polynomial $\lambda^2-2^9\lambda+2^{16}$ has double zero $\lambda=2^8$.
By the classical Poincar\'e theorem we conclude that $\limsup_{n\to\infty}|\rho_{n,i}|^{1/n}\le2^8$ for both $i=0,3$; this implies the desired claim. 
\end{proof}

\begin{lemma}\label{lem_2_adic_estimation}
We have
\[
|S_n|_2 \le 2^{-16n+o(n)} \quad\text{as}\; n \to \infty.
\]
\end{lemma}

\begin{proof}
From Definition \ref{def_R_n(t)} we have
\[
R_n\Big( t+\frac{1}{2} \Big) = 2^{12n+4} \cdot g(t) \cdot \prod_{j=1}^{n} \left( t+j \right)^4,
\quad\text{where}\;
g(t) = \frac{2t+n+1}{\prod_{j=0}^{n} \left( 2t+2j+1 \right)^4}.
\]
By the Leibniz rule we deduce
\begin{align*}
R_n'\Big( t+\frac{1}{2} \Big)
&= 2^{12n+4}  \cdot g'(t) \cdot \prod_{j=1}^{n} ( t+j)^4 \\
&\quad
+ 2^{12n+4}  \cdot g(t) \cdot 4 \cdot \sum_{j=1}^{n} (t+j)^3 \prod_{\substack{1 \le k \le n\\ k \ne j}} (t+k)^4 \\
&=2^{12n+4} n!^4 \cdot g'(t) \cdot \binom{t+n}{n}^4 \\
&\quad
+ 2^{12n+6} n!^3 \cdot g(t) \cdot \binom{t+n}{n}^3 \sum_{j=1}^{n} (j-1)! (n-j)! \binom{t+j-1}{j-1} \binom{t+n}{n-j}.
\end{align*}
By part~(d) of Lemma \ref{lem_properties_of_the_triangle_operator}, we have $\triangle(g) \ge 0$ and $\triangle(g') \ge 0$. 
By part~(e),
\begin{gather*}
\triangle\left( \binom{t+n}{n} \right) \ge -\frac{\log n}{\log 2}, \\
\triangle\left( \binom{t+j-1}{j-1} \right) \ge -\frac{\log n}{\log 2} \quad\text{and}\quad
\triangle\left( \binom{t+n}{n-j} \right) \ge -\frac{\log n}{\log 2}
\end{gather*}
for $j \in \{1,2,\dots,n\}$. 
Therefore, by part~(c) of the lemma we obtain
\begin{align*}
\triangle\bigg( g'(t) \cdot \binom{t+n}{n}^4 \bigg) &\ge -\frac{\log n}{\log 2}, \\
\triangle\bigg( g(t) \cdot \binom{t+n}{n}^3\binom{t+j-1}{j-1} \binom{t+n}{n-j}  \bigg) &\ge -\frac{\log n}{\log 2}
\end{align*}
for $j \in \{1,2,\dots,n\}$. 
Now notice that
\begin{align*}
v_2(n!) \ge n - \frac{\log (n+1)}{\log 2} \quad\text{and}\quad
v_2((j-1)!(n-j)!) \ge n - 1 - \frac{2 \log (n+1)}{\log 2}
\end{align*}
for such $j$. 
Thus, by parts (a), (b) of Lemma \ref{lem_properties_of_the_triangle_operator} we obtain
\[
\triangle\Big( R_n'\Big( t+\frac{1}{2} \Big) \Big) \ge 16n + 4 - \frac{6\log(n+1)}{\log 2}.
\]
Finally, by Definition \ref{def_Sn} and Lemma \ref{lem_estimation_of_Volkenborn_integral}, we have
\[
v_2(S_n) \ge \triangle\Big( R_n'\Big( t+\frac{1}{2} \Big) \Big) - 1
\ge 16n + 3  -\frac{6\log(n+1)}{\log 2},
\]
which implies the asymptotics claimed.
\end{proof}

\section{Proof of the main theorem}

\begin{proof}[Proof of Theorem \textup{\ref{main}}]
For any $n \in \mathbb{Z}_{>0}$, define
\[
\widehat{S}_n := \Phi_n^{-1}d_n^6 \cdot S_n.
\]
By Lemma \ref{lem_linear_forms}, Lemma \ref{lem_arith_rho} and Lemma \ref{almost_GDC} we have
\[
\widehat{S}_n = a_n + b_n \zeta_2(5),
\quad\text{where}\;\;
a_n = \Phi_n^{-1}d_n^6 \cdot \rho_{n,0} \in \mathbb{Z}, \;
b_n = \Phi_n^{-1}d_n^6 \cdot \rho_{n,3} \in \mathbb{Z}.
\]
Furthermore, by Lemma \ref{lem_rec} (b),
\[
a_n b_{n+1} - a_{n+1}b_n \neq 0 \quad \text{for all}\; n \in \mathbb{Z}_{>0}.
\]
Note that $v_2(\Phi_n^{-1}d_n^6) = O(\log n)$ as $n \to \infty$. By Lemma \ref{lem_2_adic_estimation}, we obtain
\[
|\widehat{S}_n|_2 \le \exp(-\alpha n + o(n)) \quad\text{as}\; n \to \infty,
\]
where $\alpha = 16\log 2$.
By Lemma \ref{lem_analysis} and Equations \eqref{eq:PNT}, we have
\[
\max\{ |a_n|, |b_n| \} \le \exp(\beta n + o(n)) \quad \text{as}\; n \to \infty,
\]
where $\beta = 8\log 2 + 5$.
Since $\alpha > \beta >0$, we conclude that $\zeta_2(5)$ is irrational.
Finally, from Lemma \ref{lem_Bel} we deduce the estimate
\[
\mu(\zeta_2(5)) \le \frac{\alpha}{\alpha-\beta} = \frac{16\log 2}{8\log 2 - 5} = 20.342651\dots
\]
for the irrationality measure.
\end{proof}

\section{Final remarks}

The existing provable instances of `denominator conjectures' concerning the denominators of coefficients of certain linear forms in zeta values, suggest that expectation \eqref{den-con} can be achieved in a transparent way, that is, via a different representation of linear forms from Lemma~\ref{lem_linear_forms} for which one can deduce the inclusions in~\eqref{den-con} `straight away', without separate manipulations for the coefficients $\rho_{n,0}$ and $\rho_{n,3}$ only.
Such a different representation is expected to be a multiple hypergeometric sum, at least a double sum as Lemma~\ref{lem:rho0} hints at.
It may still be a cumbersome task to make such a transparency simple\,---\,one can foretaste the potentials from a related treatment of linear forms in $1$ and $\zeta(4)$ in \cite[Sect.~3]{MZ2020}.

Our proof of the irrationality of $\zeta_2(5)$ alone can be lightened through avoiding the use of Lemma \ref{lem-p-est} (and of Lemma~\ref{lem:rho0}).
Already from Lemma~\ref{lem_arith_rho} and Lemma~\ref{lem_rec}~(a) one finds out that $\Psi_n\rho_{n,i}\in\mathbb Z$ for $n\in\mathbb Z_{>0}$ and $i\in\{0,3\}$, where $\Psi_n=\gcd(d_n^6,n!^5)$ behaves asymptotically as $\exp(5.5n+o(n))$ as $n\to\infty$.
Since $8\log 2 + 5.5 < 16\log 2$, the conclusion $\zeta_2(5)\notin\mathbb Q$ still holds but the corresponding estimate for $\mu(\zeta_2(5))$ is not worth a mention.

Speaking about the irrationality measure for $\zeta_2(5)$ and also for the related $p$-adic zeta values $\zeta_2(3),\zeta_3(3)$, our construction in this note and the earlier ones in \cite{Beu2008,Cal2005,Lai2025} exploit exclusively the so-called totally symmetric hypergeometric approximations (or their equivalents).
It may certainly be of interest to take a step further and investigate more general hypergeometric series leading to rational approximations that depend not only on $n\in\mathbb Z_{\ge0}$ but on multiple integer parameters.
Such generalisations are amenable to additional arithmetic manipulations, for example, to the use of the arithmetic group method as developed in the works of Rhin and Viola \cite{RV2001}; see further examples in \cite{BZ2022+,MZ2020}.

In a different direction, the recursion \eqref{eq:rec} is ultimately linked with the theory of Calabi--Yau differential operators \cite{AZ2006,vSt2018} and\,---\,potentially\,---\,with geometry of Calabi--Yau three- and fourfolds. Algebraic varieties whose periods are associated with irrationality proofs are known to be particularly `minimalistic' (for example, in the sense of level or conductor)\,---\,one can find recent examples of such connection in \cite{vSt2021}.

Finally, it seems evident that the Volkenborn integral of a rational function is a hypergeometric object deserving investigation on its own, especially from the point of view of summation and transformation formulae. At the moment, manipulations with linear forms in the values of $p$-adic zeta functions are performed hypergeometrically at the level of their coefficients, as the forms themselves are not compatible with their Archimedean `traditional-hypergeometric' companions\,---\,our Remark~\ref{rem_diff_arch} illustrates well this discrepancy.
On the other hand, one can show that the two integrals
\begin{alignat*}{2}
S_n^{(\mathrm{L})}
&:= -\int_{\mathbb{Z}_2} R_n^{(\mathrm{L})}(t+\tfrac{1}{4}) \,\d t
=\rho_{n,0}^{(\mathrm{L})} + \rho_{n,2}^{(\mathrm{L})} \zeta_2(3),
&\quad\text{where}\;
R_n^{(\mathrm{L})}(t) &:= 2^{6n} \frac{(t+3/4)_n^2}{(t)_{n+1}^2}
\\ \intertext{(the special case $s=0$ of \cite[Theorem 1.5]{Lai2025}), and}
S_n^{(\mathrm{B})}
&:= -\int_{\mathbb{Z}_2} R_n^{(\mathrm{B})}(t+\tfrac{1}{2}) \,\d t
=\rho_{n,0}^{(\mathrm{B})} + \rho_{n,2}^{(\mathrm{B})} \zeta_2(3),
&\quad\text{where}\;
R_n^{(\mathrm{B})}(t) &:= 2^{6n} (2t+n) \frac{(t+1/2)_n^3}{(t)_{n+1}^3},
\end{alignat*}
correspond to the \emph{same} linear forms in $1$ and $\zeta_2(3)$.
(These linear forms also coincide essentially with those constructed in \cite{Cal2005} and \cite{Beu2008} by other means.)
It is natural to expect a hypergeometric-type identity for the Volkenborn integrals behind the coincidence $S_n^{(\mathrm{L})}=S_n^{(\mathrm{B})}$.

\begin{acknowledgements}
We are indebted to the anonymous referee for their enthusiastic report on this note and creative feedback.
\begin{otherlanguage}{vietnamese}
The first author thanks Di{\~\ecircumflex}m My for inspiring him at H{\`\ocircumflex} Ch{\'i} Minh City.
He is supported by Research Foundation for Scholars of Xiamen University X2450218.
\end{otherlanguage}
The second author  gratefully acknowledges the support through the DFG funded Collaborative Research Center SFB 1085 'Higher Invariants'.
The third author thanks the Institut des Hautes \'Etudes Scientifiques (Bures-sur-Yvette, France) and the Max-Planck Institute for Mathematics (Bonn, Germany) for hospitality and support during his stays in March 2025 and April--June 2025, respectively; his work on this project was fully performed during these stays.
\end{acknowledgements}

\vspace*{3mm}
\begin{flushright}
\begin{minipage}{148mm}\sc\footnotesize
L.\,L.: School of Mathematical Sciences, Xiamen University, Fujian, China \\
{\it E-mail address}: \href{mailto:lilaimath@gmail.com}{{\tt lilaimath@gmail.com}} \vspace*{3mm}
\end{minipage}
\end{flushright}

\begin{flushright}
\begin{minipage}{148mm}\sc\footnotesize
J.\,S.: Department of Mathematics, University of Duisburg-Essen, Essen, Germany \\
{\it E-mail address}: \href{mailto:johannes.sprang@uni-due.de}{{\tt johannes.sprang@uni-due.de}} \vspace*{3mm}
\end{minipage}
\end{flushright}

\begin{flushright}
\begin{minipage}{148mm}\sc\footnotesize
W.\,Z.: IMAPP, Radboud University Nijmegen, The Netherlands \& \\ Max-Planck Institute for Mathematics, Bonn, Germany\\
{\it E-mail address}: \href{mailto:w.zudilin@math.ru.nl}{{\tt w.zudilin@math.ru.nl}} \vspace*{3mm}
\end{minipage}
\end{flushright}

\end{document}